\documentclass[12pt]{amsart}
\usepackage{amsmath,amsthm,amsfonts,amssymb,mathrsfs}
\date{\today}

\usepackage{color}
\input xymatrix
\xyoption{all}

\usepackage{hyperref}

\newtheorem{theorem}{Theorem}[section]

\newtheorem{proposition}[theorem]{Proposition}
\newtheorem{corollary}[theorem]{Corollary}
\newtheorem{lemma}[theorem]{Lemma}
\theoremstyle{definition}
\newtheorem{example}[theorem]{Example}%[section]
\newtheorem{remark}[theorem]{Remark}%[section]

\newtheorem{definition}[theorem]{Definition}%[section]

\begin{document}

\title[Pseudocompactness, products and topological Brandt $\lambda^0$-extensions ...]{Pseudocompactness, products and topological Brandt $\lambda^0$-extensions of semitopological monoids}

\author[O.~Gutik]{Oleg~Gutik}
\address{Faculty of Mathematics, National University of Lviv,
Universytetska 1, Lviv, 79000, Ukraine}
\email{o\underline{\hskip5pt}\,gutik@franko.lviv.ua,
ovgutik@yahoo.com}

\author[O.~Ravsky]{Oleksandr~Ravsky}
\address{Pidstrygach Institute for Applied Problems of Mechanics and Mathematics of NASU, Naukova 3b, Lviv, 79060, Ukraine}
\email{oravsky@mail.ru}



\begin{abstract}
In the paper we study the preservation
of pseudocompactness (resp., countable compactness, sequential compactness, $\omega$-boundedness, totally countable compactness, countable pracompactness, sequential pseudocompactness)
by Tychonoff products of pseudocompact (and countably compact) to\-pological Brandt $\lambda_i^0$-extensions of semitopological monoids with zero.
In particular we show that if $\big\{ \big(B^0_{\lambda_i}(S_i),\tau^0_{B(S_i)}\big) \colon i\in\mathscr{I}\big\}$ is a family of Hausdorff pseudocompact to\-pological Brandt $\lambda_i^0$-extensions of pseudocompact semitopological monoids with zero such that the Tychonoff product $\prod\left\{ S_i \colon i\in\mathscr{I}\right\}$ is a pseudocompact space then the direct product $\prod\big\{ \big(B^0_{\lambda_i}(S_i),\tau^0_{B(S_i)}\big) \colon i\in\mathscr{I}\big\}$ endowed with the Tychonoff topology is a Hausdorff pseudocompact semitopological semigroup.
\end{abstract}

\maketitle

\section{Introduction and preliminaries}

Further we shall follow the terminology of \cite{CHK, CliffordPreston1961-1967, Engelking1989, Petrich1984, Ruppert1984}. By $\mathbb{N}$ we shall denote the set of all positive integers.

A semigroup is a non-empty set with a binary associative operation. A semigroup $S$ is called \emph{inverse} if for any $x\in S$ there exists a unique $y\in S$ such that $x\cdot y\cdot x=x$ and $y\cdot x\cdot y=y$. Such the element $y$ in $S$ is called
\emph{inverse} to $x$ and is denoted by $x^{-1}$. The map assigning to each element $x$ of an inverse semigroup $S$ its inverse $x^{-1}$ is called the \emph{inversion}.

For a semigroup $S$ by $E(S)$ we denote the subset of idempotents of $S$, and by $S^1$ (resp., $S^0$) we denote the semigroup $S$ the adjoined unit (resp., zero) (see \cite[Section~1.1]{CliffordPreston1961-1967}). Also if a
semigroup $S$ has zero $0_S$, then for any $A\subseteq S$ we
denote $A^*=A\setminus\{ 0_S\}$.

For  a semilattice $E$  the semilattice operation on $E$ determines the partial order $\leqslant$ on $E$: $$e\leqslant f\quad\text{if and only if}\quad ef=fe=e.$$ This order is called {\em natural}. An element $e$ of a partially ordered set $X$ is
called {\em minimal} if $f\leqslant e$  implies $f=e$ for $f\in X$. An idempotent $e$ of a semigroup $S$ without zero (with zero) is called \emph{primitive} if $e$ is a minimal element in $E(S)$ (in $(E(S))^*$).

Let $S$ be a semigroup with zero and $\lambda\geqslant 1$ be a cardinal. On the set
$B_{\lambda}(S)=\left(\lambda\times S\times\lambda\right)\sqcup\{ 0\}$ we define a  semigroup operation as follows
 $$
 (\alpha,a,\beta)\cdot(\gamma, b, \delta)=
  \begin{cases}
    (\alpha, ab, \delta), & \text{ if } \beta=\gamma; \\
    0, & \text{ if } \beta\ne \gamma,
  \end{cases}
 $$
and $(\alpha, a, \beta)\cdot 0=0\cdot(\alpha, a, \beta)=0\cdot 0=0$, for all $\alpha, \beta, \gamma, \delta\in \lambda$ and $a, b\in S$. If $S$ is a monoid, then the semigroup $B_\lambda(S)$ is called the {\it Brandt $\lambda$-extension of the semigroup}
$S$~\cite{Gutik1999}. Obviously, ${\mathcal J}=\{ 0\}\cup\{(\alpha, {\mathscr O}, \beta)\colon {\mathscr O}$ is the zero of $S\}$ is an ideal of $B_\lambda(S)$. We put $B^0_\lambda(S)=B_\lambda(S)/{\mathcal J}$ and we shall call $B^0_\lambda(S)$ the {\it Brandt $\lambda^0$-extension of the semigroup $S$ with zero}~\cite{GutikPavlyk2006}. Further, if $A\subseteq S$ then we shall denote $A_{\alpha,\beta}=\{(\alpha, s, \beta)\colon s\in A \}$ if $A$ does not contain zero, and $A_{\alpha,\beta}=\{(\alpha, s, \beta)\colon s\in A\setminus\{ 0\} \}\cup \{ 0\}$ if $0\in A$, for $\alpha, \beta\in {\lambda}$. If $\mathcal{I}$ is a trivial  semigroup (i.e., $\mathcal{I}$ contains only one element), then by ${\mathcal{I}}^0$ we denote the semigroup $\mathcal{I}$ with the adjoined zero. Obviously, for any
$\lambda\geqslant 2$ the Brandt $\lambda^0$-extension of the semigroup ${\mathcal{I}}^0$ is isomorphic to the semigroup of $\lambda\times\lambda$-matrix units and any Brandt $\lambda^0$-extension of a semigroup with zero contains the
semigroup of $\lambda\times\lambda$-matrix units. Further by $B_\lambda$ we shall denote the semigroup of $\lambda\times\lambda$-matrix units and by $B^0_\lambda(1)$  the subsemigroup of $\lambda\times\lambda$-matrix units of the Brandt $\lambda^0$-extension of a monoid $S$ with zero.

A semigroup $S$ with zero is called \emph{$0$-simple} if $\{0\}$ and $S$ are its only ideals and $S^2\neq\{0\}$, and \emph{completely $0$-simple} if it is $0$-simple and has a primitive idempotent~\cite{CliffordPreston1961-1967}. A completely $0$-simple inverse semigroup is called a \emph{Brandt semigroup}~\cite{Petrich1984}. By  Theorem~II.3.5~\cite{Petrich1984}, a semigroup $S$ is a Brandt semigroup if and only if $S$ is isomorphic to a Brandt $\lambda$-extension $B_\lambda(G)$ of a group  $G$.

A non-trivial inverse semigroup is called a \emph{primitive inverse semigroup} if all its non-zero idempotents are primitive~\cite{Petrich1984}. A semigroup $S$ is a primitive inverse semigroup if and only if $S$ is an orthogonal sum of Brandt semigroups~\cite[Theorem~II.4.3]{Petrich1984}.

In this paper all topological spaces are Hausdorff. If $Y$ is a subspace of a topological space $X$ and $A\subseteq Y$, then by $\operatorname{cl}_Y(A)$ and $\operatorname{int}_Y(A)$ we denote the topological closure and interior of $A$ in $Y$, respectively.

A subset $A$ of a topological space $X$ is called \emph{regular open} if $\operatorname{int}_X(\operatorname{cl}_X(A))=A$.

We recall that a topological space $X$ is said to be
\begin{itemize}
  \item \emph{semiregular} if $X$ has a base consisting of regular open subsets;
  \item \emph{compact} if each open cover of $X$ has a finite subcover;
  \item \emph{sequentially compact} if each sequence $\{x_i\}_{i\in\mathbb{N}}$ of $X$ has a convergent subsequence in $X$;
  \item \emph{$\omega$-bounded} if every countably infinite set in $X$ has the compact closure~\cite{GuldenFleischmanWeston1970};
  \item \emph{totally countably compact} if every countably infinite set in $X$ contains an infinite subset with the compact closure~\cite{Frolik1960};
  \item \emph{countably compact} if each open countable cover of $X$ has a finite subcover;
  \item \emph{countably compact at a subset} $A\subseteq X$ if every infinite subset $B\subseteq A$ has an accumulation  point $x$ in $X$;
  \item \emph{countably pracompact} if there exists a dense subset $A$ in $X$  such that $X$ is countably compact at $A$~\cite{Arkhangelskii1992};
  \item \emph{sequentially pseudocompact} if for each sequence $\{U_n\colon n\in\mathbb{N}\}$ of non-empty open subsets of the space $X$ there exist a point $x\in X$ and an infinite set $S\subset\mathbb{N}$ such that for each neighborhood $U$ of the point $x$ the set $\{n\in S\colon U_n\cap U=\varnothing\}$ is finite \cite{GutikRavsky2014??};
  \item $H$-\emph{closed} if $X$ is Hausdorff and $X$ is a closed subspace of every Hausdorff space in which it is contained \cite{AlexandroffUrysohn1929};
  \item \emph{pseudocompact} if each locally finite open cover of $X$ is finite.
  %\item $k$-\emph{space} if a subset $F\subset X$ is closed in $X$ if and only if $F\cap K$ is closed in $K$ for every compact subspace $K\subseteq X$.
\end{itemize}
According to Theorem~3.10.22 of \cite{Engelking1989}, a Tychonoff topological space $X$ is pseudocompact if and only if each continuous real-valued function on $X$ is bounded. Also, a Hausdorff topological space $X$ is pseudocompact if and only if every locally finite family of non-empty open subsets of $X$ is finite. Every compact space and every sequentially compact space are countably compact, every countably compact space is countably pracompact, and every countably pracompact space is pseudocompact (see \cite{Arkhangelskii1992}). We observe that pseudocompact spaces in topological literature also are called  \emph{lightly  compact} or \emph{feebly compact} (see \cite{Bagley-Connelly-McKnight-1958, Fernandez-Tkachenko-2014, Sanchis-Tkachenko-2012}).

We recall that the Stone-\v{C}ech compactification of a Tychonoff space $X$ is a
compact Hausdorff space $\beta X$ containing $X$ as a dense subspace so that each continuous map $f\colon X\rightarrow Y$ to a compact Hausdorff space $Y$ extends to a continuous map $\overline{f}\colon \beta X\rightarrow Y$ \cite{Engelking1989}.

A ({\it semi})\emph{topological semigroup} is a Hausdorff topological space with a (separately) continuous semigroup operation. A topological semigroup which is an inverse semigroup is called an \emph{inverse topological semigroup}. A \emph{topological inverse semigroup} is an inverse topological semigroup with continuous inversion. We observe that the inversion on a topological inverse semigroup is a homeomorphism (see \cite[Proposition~II.1]{EberhartSelden1969}). A Hausdorff topology $\tau$ on a (inverse) semigroup $S$ is called (\emph{inverse}) \emph{semigroup} if $(S,\tau)$ is a topological (inverse) semigroup. A {\it paratopological} (\emph{semitopological}) \emph{group} is a Hausdorff topological space with a jointly (separately) continuous group operation. A paratopological group with continuous inversion is a \emph{topological group}.

Let $\mathfrak{STSG}_0$ be a class of semitopological semigroups.

\begin{definition}[\cite{Gutik1999}]\label{def1}
Let $\lambda\geqslant 1$ be a cardinal and $(S,\tau)\in\mathfrak{STSG}_0$ be a semitopological monoid with zero. Let $\tau_{B}$ be a topology on $B_{\lambda}(S)$ such that
\begin{itemize}
    \item[a)] $\left(B_{\lambda}(S),\tau_{B}\right)\in \mathfrak{STSG}_0$; \; and
    \item[b)] for some $\alpha\in{\lambda}$ the topological subspace $(S_{\alpha,\alpha},\tau_{B}|_{S_{\alpha,\alpha}})$ is naturally homeomorphic to $(S,\tau)$.
\end{itemize}
Then $\left(B_{\lambda}(S), \tau_{B}\right)$ is called a {\it topological Brandt $\lambda$-extension of $(S, \tau)$ in $\mathfrak{STSG}_0$}.
\end{definition}

\begin{definition}[\cite{GutikPavlyk2006}]\label{def2}
Let $\lambda\geqslant 1$ be a cardinal and $(S,\tau)\in\mathfrak{STSG}_0$. Let $\tau_{B}$ be a topology on $B^0_{\lambda}(S)$ such that
\begin{itemize}
  \item[a)] $\left(B^0_{\lambda}(S),            \tau_{B}\right)\in\mathfrak{STSG}_0$;
  \item[b)] the topological subspace $(S_{\alpha,\alpha},\tau_{B}|_{S_{\alpha,\alpha}})$ is naturally homeomorphic to $(S,\tau)$ for some $\alpha\in{\lambda}$.
\end{itemize}
Then $\left(B^0_{\lambda}(S), \tau_{B}\right)$ is called a {\it topological Brandt $\lambda^0$-extension of $(S, \tau)$ in $\mathfrak{STSG}_0$}.
\end{definition}

Later, if $\mathfrak{STSG}_0$ coincides with the class of all semitopological semigroups we shall say that $\left(B^0_{\lambda}(S), \tau_{B}\right)$ (resp.,  $\left(B_{\lambda}(S), \tau_{B}\right)$) is called a {\it topological Brandt $\lambda^0$-extension} (resp., a \emph{topological Brandt $\lambda$-extension}) of $(S, \tau)$.

Algebraic properties of Brandt $\lambda^0$-extensions of monoids with zero, non-trivial homomorphisms between them, and a category whose objects are ingredients of the construction of such extensions were described in \cite{GutikRepovs2010}. Also, in  \cite{GutikPavlykReiter2009} and \cite{GutikRepovs2010} a category whose objects are ingredients in the constructions of finite (resp., compact, countably compact) topological Brandt $\lambda^0$-extensions of topological monoids with zeros were described.

Gutik and Repov\v{s} proved that any $0$-simple countably compact topological inverse semigroup is topologically isomorphic to a topological Brandt $\lambda$-extension $B_{\lambda}(H)$ of a countably compact topological group $H$ in the class of all topological inverse semigroups for some finite cardinal $\lambda\geqslant 1$ \cite{GutikRepovs2007}.  Also, every $0$-simple pseudocompact topological inverse semigroup is topologically isomorphic to a topological Brandt $\lambda$-extension $B_{\lambda}(H)$ of a pseudocompact topological group $H$ in the class of all topological inverse semigroups for some finite cardinal $\lambda\geqslant 1$ \cite{GutikPavlykReiter2011}. Next Gutik and Repov\v{s} showed in \cite{GutikRepovs2007} that the Stone-\v{C}ech compactification $\beta(T)$ of a $0$-simple countably compact topological inverse semigroup $T$ has a natural structure of a $0$-simple compact topological inverse semigroup. It was proved in \cite{GutikPavlykReiter2011} that the same is true for $0$-simple pseudocompact topological inverse semigroups.

In the paper \cite{BerezovskiGutikPavlyk2010} the structure of compact and countably compact primitive topological inverse semigroups was described and was showed that any countably compact primitive topological inverse semigroup embeds into a compact primitive topological inverse semigroup.

Comfort and Ross in \cite{ComfortRoss1966} proved that a Tychonoff product of an arbitrary non-empty family of pseudocompact topological groups is a pseudocompact topological group. Also, they proved there that the Stone-\v{C}ech compactification of a pseudocompact topological group has a natural structure of a compact topological group. Ravsky in \cite{Ravsky-arxiv1003.5343v5} generalized Comfort--Ross Theorem and proved that a Tychonoff product of an arbitrary non-empty family of pseudocompact paratopological groups is pseudocompact.

In the paper \cite{GutikPavlyk2013} it is described the structure of pseudocompact primitive topological inverse semigroups and it is shown that the Tychonoff product of an arbitrary non-empty family of pseudocompact primitive topological inverse semigroups is pseudocompact. Also, there is proved that the Stone-\v{C}ech compactification of a pseudocompact primitive topological inverse semigroup has a natural structure of a compact primitive topological inverse semigroup.

In the paper \cite{GutikRavsky2013??} we studied the structure of inverse primitive pseudocompact semitopological and topological semigroups. We find conditions when a maximal subgroup of an inverse primitive pseudocompact semitopological semigroup $S$ is a closed subset of $S$ and described the topological structure of such semiregular semigroup. Also there we described structure of pseudocompact topological Brandt $\lambda^0$-extensions of topological semigroups and semiregular (quasi-regular) primitive inverse topological semigroups. In \cite{GutikRavsky2013??} we shown that the inversion in a quasi-regular primitive inverse pseudocompact topological semigroup is continuous. Also there, an analogue of Comfort--Ross Theorem is proved for such semigroups: the Tychonoff product of an arbitrary non-empty family of primitive inverse semiregular pseudocompact semitopological semigroups with closed maximal subgroups is a pseudocompact space, and we described the structure of the Stone-\v{C}ech compactification of a Hausdorff primitive inverse countably compact semitopological semigroup $S$ such that every maximal subgroup of $S$ is a topological group.

In this paper we study the preserving of Tychonoff products of the pseudocompactness (resp., countable compactness, sequential compactness, $\omega$-boundedness, totally countable compactness, countable pracompactness, sequential pseudocompactness) by pseudocompact (and countably compact) to\-pological Brandt $\lambda_i^0$-extensions  of semitopological semitopological monoids with zero.
In particular we show that if $\big\{ \big(B^0_{\lambda_i}(S_i),\tau^0_{B(S_i)}\big) \colon i\in\mathscr{I}\big\}$ is a family of Hausdorff pseudocompact to\-pological Brandt $\lambda_i^0$-extension of pseudocompact semitopological monoids with zero such that the Tychonoff product $\prod\left\{ S_i \colon i\in\mathscr{I}\right\}$ is a pseudocompact space, then the direct product $\prod\big\{ \big(B^0_{\lambda_i}(S_i),\tau^0_{B(S_i)}\big) \colon i\in\mathscr{I}\big\}$ with the Tychonoff topology is a Hausdorff pseudocompact semitopological semigroup.
%%%%%%%%%%%%%%%%%%%%%%%%%%%%%%%%%%%%%%%%%%%%%%%%%%%%%

\section{Tychonoff products of pseudocompact topological Brandt $\lambda^0$-extensions of semitopological semigroups}

Later we need the following theorem from \cite{GutikPavlyk2013a}:

\begin{theorem}[{\cite[Theorem~12]{GutikPavlyk2013a}}]\label{theorem-2.1}
For any Hausdorff countably compact semitopological monoid $(S,\tau)$ with zero and for any cardinal $\lambda\geqslant 1$ there exists a unique Hausdorff countably compact  topological
Brandt $\lambda^0$-extension $\left(B^0_{\lambda}(S),\tau_{B}^S\right)$ of $(S,\tau)$ in the
class of semitopological semigroups, and the topology $\tau_{B}^S$
is generated by the base
$\mathscr{B}_B=\bigcup\left\{\mathscr{B}_B(t)\colon t\in
B^0_{\lambda}(S)\right\}$, where:
\begin{itemize}
    \item[$(i)$] $\mathscr{B}_B(t)=\big\{(U(s) \setminus\{ 0_S\})_{\alpha,\beta} \colon U(s)\in \mathscr{B}_S(s)\big\}$, where $t=(\alpha,s,\beta)$ is a non-zero element of    $B^0_{\lambda}(S)$, $\alpha,\beta\in\lambda$;

    \item[$(ii)$] $\mathscr{B}_B(0)=\Big\{U_{A}(0)= \bigcup_{(\alpha,\beta)\in(\lambda\times\lambda) \setminus A}S_{\alpha,\beta}\cup \bigcup_{(\gamma,\delta)\in A} (U(0_S))_{\gamma,\delta} \colon A \hbox{~is a finite subset of~} \lambda\times\lambda$ and  $U(0_S)\in\mathscr{B}_S(0_S)\Big\}$, where $0$ is the zero of $B^0_{\lambda}(S)$,
\end{itemize}
and $\mathscr{B}_S(s)$ is a base of the topology $\tau$ at the point $s\in S$.
\end{theorem}

\begin{lemma}\label{lemma-2.2}
For any Hausdorff sequentially compact semitopological  monoid
$(S,\tau)$ with zero and for any cardinal $\lambda\geqslant 1$
the Hausdorff countably compact  topological
Brandt $\lambda^0$-extension
$\left(B^0_{\lambda}(S),\tau_{B}^S\right)$ of $(S,\tau)$ in the
class of semitopological semigroups is a sequentially compact space.
\end{lemma}

\begin{proof}
In the case when $\lambda<\omega$ the statement of the lemma follows from Theorems~3.10.32 and 3.10.34 from \cite{Engelking1989}.

Next we suppose that $\lambda\geqslant\omega$. Let $\mathscr{A}(\lambda)$ be the one point Alexandroff compactification of the discrete space of cardinality $\lambda$. Then $\mathscr{A}(\lambda)$ is scattered because $\mathscr{A}(\lambda)$ has only one non-isolated point, and hence by Theorem~5.7 from \cite{VaughanHSTT} the space $\mathscr{A}(\lambda)$ is sequentially compact. Since cardinal $\lambda$ is infinite without loss of generality we can assume that $\lambda=\lambda\cdot\lambda$ and hence we can identify the space $\mathscr{A}(\lambda)$ with $\mathscr{A}(\lambda\times\lambda)$. Then by Theorem~3.10.35 from \cite{Engelking1989} the space $\mathscr{A}(\lambda\times\lambda)\times S$ is sequentially compact. Later we assume that $a$ is non-isolated point of the space $\mathscr{A}(\lambda\times\lambda)$. We define the map $g\colon\mathscr{A}(\lambda\times\lambda)\times S\to B^0_{\lambda}(S)$ by the formulae
\begin{equation*}
    g(a)=0 \qquad \hbox{and} \qquad g((\alpha,\beta,s))=
\left\{
  \begin{array}{cl}
    (\alpha,s,\beta), & \hbox{if~} s\in S\setminus\{0_s\};\\
    0, & \hbox{if~} s=0_S.
  \end{array}
\right.
\end{equation*}
Theorem~\ref{theorem-2.1} implies that so defined map $g$ is continuous and hence by Theorem~3.10.32 of \cite{Engelking1989} we get that the topological Brandt $\lambda^0$-extension $\left(B^0_{\lambda}(S),\tau_{B}^S\right)$ of $(S,\tau)$ in the class of semitopological semigroups is a sequentially compact space.
\end{proof}

Lemma~\ref{lemma-2.2} and Theorem~3.10.35 from \cite{Engelking1989} imply the following theorem:

\begin{theorem}\label{theorem-2.3}
Let $\left\{B^0_{\lambda_i}(S_i) \colon i\in\omega\right\}$ be a countable family of Hausdorff countably compact topological Brandt $\lambda^0_i$-extension of sequentially compact Hausdorff semitopological monoids. Then the direct product $\prod\left\{B^0_{\lambda_i}(S_i)\colon i\in\omega\right\}$ with the Tychonoff topology is a Hausdorff sequentially compact semitopological semigroup.
\end{theorem}

\begin{theorem}\label{theorem-2.4}
Let $\left\{ B^0_{\lambda_i}(S_i) \colon i\in\mathscr{I}\right\}$ be a non-empty family of Hausdorff countably compact topological Brandt $\lambda^0_i$-extension of countably compact Hausdorff semitopological monoids such that the Tychonoff product $\prod\left\{ S_i \colon i\in\mathscr{I}\right\}$ is a countably compact space. Then the direct product  $\prod\left\{ B^0_{\lambda_i}(S_i) \colon i\in\mathscr{I}\right\}$ with the Tychonoff topology is a Hausdorff countably compact semitopological semigroup.
\end{theorem}

\begin{proof}
For every infinite cardinal $\lambda_i$, $i\in\mathscr{I}$, we shall repeat the construction proposed in the proof of Lemma~\ref{lemma-2.2}. Let $\mathscr{A}(\lambda_i)$ be the one point Alexandroff compactification of the discrete space of cardinality $\lambda_i$. Since cardinal $\lambda_i$ is infinite without loss of generality we can assume that $\lambda_i=\lambda_i\cdot\lambda_i$ and hence we can identify the space $\mathscr{A}(\lambda_i)$ with $\mathscr{A}(\lambda_i\times\lambda_i)$. Later we assume that $a_i$ is non-isolated point of the space $\mathscr{A}(\lambda_i\times\lambda_i)$. We define the map $g_i\colon\mathscr{A}(\lambda_i\times\lambda_i)\times S_i\to B^0_{\lambda_i}(S_i)$ by the formulae
\begin{equation}\label{eq-1}
    g_i(a_i)=0_i \qquad \hbox{and} \qquad g_i((\alpha_i,\beta_i,s_i))=
\left\{
  \begin{array}{cl}
    (\alpha_i,s_i,\beta_i), & \hbox{if~} s_i\in S\setminus\{0_{S_i}\};\\
    0_i, & \hbox{if~} s=0_{S_i},
  \end{array}
\right.
\end{equation}
where $0_i$ and $0_{S_i}$ are zeros of the semigroup $B^0_{\lambda_i}(S_i)$ and the monoid $S_i$, respectively. Theorem~\ref{theorem-2.1} implies that so defined map $g_i$ is continuous.

In the case when cardinal $\lambda_i$, $i\in\mathscr{I}$, is finite we put $\mathscr{A}(\lambda_i\times\lambda_i)$ is the discrete space of cardinality $\lambda_i^2+1$ with the fixed point $a_i\in \mathscr{A}(\lambda_i\times\lambda_i)$. Next we define the map $g_i\colon\mathscr{A}(\lambda_i\times\lambda_i)\times S_i\to B^0_{\lambda_i}(S_i)$ by the formulae (\ref{eq-1}), where $0_i$ and $0_{S_i}$ are zeros of the semigroup $B^0_{\lambda_i}(S_i)$ and the monoid $S_i$, respectively. It is obviously that such defined map $g_i$ is continuous. Then the space $\prod_{i\in\mathscr{I}}\mathscr{A}(\lambda_i\times\lambda_i)\times S_i$ is homeomorphic to $\prod_{i\in\mathscr{I}}\mathscr{A}(\lambda_i\times\lambda_i)\times \prod_{i\in\mathscr{I}} S_i$ and hence by Theorem~3.2.4 and Corollary~3.10.14 from \cite{Engelking1989} the Tychonoff product $\prod_{i\in\mathscr{I}}\mathscr{A}(\lambda_i\times\lambda_i)\times S_i$ is countably compact. Later we define the map $g\colon \prod_{i\in\mathscr{I}} \mathscr{A}(\lambda_i\times\lambda_i)\times S_i\to \prod_{i\in\mathscr{I}} B^0_{\lambda_i}(S_i)$ by putting $g= \prod_{i\in\mathscr{I}}g_i$. Since for any $i\in\mathscr{I}$ the map $g_i\colon\mathscr{A}(\lambda_i\times\lambda_i)\times S_i\to B^0_{\lambda_i}(S_i)$ is continuous, Theorem~\ref{theorem-2.1} and Proposition~2.3.6 of \cite{Engelking1989} imply that $g$ is continuous too. Therefore by Theorem~3.10.5 from \cite{Engelking1989} we obtain that the direct product $\prod\{B^0_{\lambda_i}(S_i)\colon i\in\mathscr{I}\}$ with the Tychonoff topology is a Hausdorff countably compact semitopological semigroup.
\end{proof}

\begin{lemma}\label{lemma-2.5}
For any Hausdorff totally countably compact semitopological monoid $(S,\tau)$ with zero and for any cardinal $\lambda\geqslant 1$ the Hausdorff countably compact  topological Brandt $\lambda^0$-extension $\left(B^0_{\lambda}(S),\tau_{B}^S\right)$ of $(S,\tau)$ in the class of semitopological semigroups is a totally countably compact space.
\end{lemma}

\begin{proof}
In the case when $\lambda<\omega$ the statement of the lemma is trivial. So we suppose that $\lambda\geqslant\omega$.

Let $A$ be an arbitrary countably infinite subset of $\left(B^0_{\lambda}(S),\tau_{B}^S\right)$. Put $\mathscr{J}=\{(\alpha,\beta)\in\lambda\times\lambda\colon A\cap S_{\alpha,\beta}\ne\emptyset\}$. If the set $\mathscr{J}$ is finite then total countable compactness of the space $(S,\tau)$ and Lemma~2 of \cite{GutikPavlyk2013a} imply the statement of the lemma. So we suppose that the set $\mathscr{J}$ is infinite. For each pair of indices $(\alpha,\beta)\in \mathscr{J}$ choose
a point $a_{\alpha,\beta}\in A\cap S_{\alpha,\beta}$ and put $K=\{0\}\cup \{a_{\alpha,\beta}\colon (\alpha,\beta)\in \mathscr{J}\}$. Then the definition of the topology $\tau_{B}^S$ on $B^0_{\lambda}(S)$ implies that $K$ is a compact subset of the $\left(B^0_{\lambda}(S),\tau_{B}^S\right)$ and $K\cap A$ is infinite. This completes the proof of the lemma.
\end{proof}

Lemma~\ref{lemma-2.5} and Theorem~4.3 from \cite{Frolik1960} imply the following theorem:

\begin{theorem}\label{theorem-2.6}
Let $\left\{B^0_{\lambda_i}(S_i) \colon i\in\omega\right\}$ be a countable family of Hausdorff countably compact topological Brandt $\lambda^0_i$-extension of totally countably compact Hausdorff semitopological monoids. Then the direct product $\prod\left\{B^0_{\lambda_i}(S_i)\colon i\in\omega\right\}$ with the Tychonoff topology is a Hausdorff totally countably compact semitopological semigroup.
\end{theorem}

\begin{theorem}\label{theorem-2.7}
Let $\left\{ B^0_{\lambda_i}(S_i) \colon i\in\mathscr{I}\right\}$ be a non-empty family of Hausdorff countably compact topological Brandt $\lambda^0_i$-extension of Hausdorff totally countably compact semitopological monoids such that the Tychonoff product $\prod\left\{ S_i \colon i\in\mathscr{I}\right\}$ is a totally countably compact space. Then the direct product $\prod\left\{ B^0_{\lambda_i}(S_i) \colon i\in\mathscr{I}\right\}$ with the Tychonoff topology is a totally countably compact semitopological semigroup.
\end{theorem}

\begin{proof}
Let for every $i\in\mathscr{I}$, $\mathscr{A}(\lambda_i\times\lambda_i)$ be a space and $g_i\colon\mathscr{A}(\lambda_i\times\lambda_i)\times S_i\to B^0_{\lambda_i}(S_i)$ be a map defined in the proof of Theorem~\ref{theorem-2.4}. Also, Theorem~\ref{theorem-2.1} implies that the map $g_i$ is continuous for every $i\in\mathscr{I}$. Since the space $\prod_{i\in\mathscr{I}}\mathscr{A}(\lambda_i\times\lambda_i)\times S_i$ is homeomorphic to $\prod_{i\in\mathscr{I}}\mathscr{A}(\lambda_i\times\lambda_i)\times \prod_{i\in\mathscr{I}} S_i$ and Theorem~4.3 from \cite{Frolik1960} we see that the Tychonoff product $\prod_{i\in\mathscr{I}}\mathscr{A}(\lambda_i\times\lambda_i)\times S_i$ is a totally countably compact space. Then by Theorem~\ref{theorem-2.1} and Proposition~2.3.6 of \cite{Engelking1989} the map $g\colon \prod_{i\in\mathscr{I}}\mathscr{A}(\lambda_i\times\lambda_i)\times S_i\to \prod_{i\in\mathscr{I}}B^0_{\lambda_i}(S_i)$ defined by the formula $g= \prod_{i\in\mathscr{I}}g_i$ is continuous. Simple verification imply that a continuous image of a totally countably compact space is a totally countably compact space too. Hence the direct product $\prod\{B^0_{\lambda_i}(S_i)\colon i\in\mathscr{I}\}$ with the Tychonoff topology is a totally countably compact semitopological semigroup.
\end{proof}

Similarly to the proof of Lemma~\ref{lemma-2.5} we can prove the following

\begin{lemma}\label{lemma-2.8}
For any Hausdorff $\omega$-bounded semitopological monoid $(S,\tau)$ with zero and for any cardinal $\lambda\geqslant 1$ the Hausdorff countably compact  topological Brandt $\lambda^0$-extension $\left(B^0_{\lambda}(S),\tau_{B}^S\right)$ of $(S,\tau)$ in the class of semitopological semigroups is an $\omega$-bounded space.
\end{lemma}

Since by Lemma~4 of \cite{GuldenFleischmanWeston1970} the Tychonoff product of an arbitrary non-empty family of $\omega$-bounded spaces is an $\omega$-bounded space, similarly to the proof of Theorem~\ref{theorem-2.7} we can prove the following

\begin{theorem}\label{theorem-2.9}
Let $\left\{ B^0_{\lambda_i}(S_i) \colon i\in\mathscr{I}\right\}$ be a non-empty family of Hausdorff countably compact topological Brandt $\lambda^0_i$-extension of Hausdorff $\omega$-bounded semitopological monoids. Then the direct product $\prod\left\{ B^0_{\lambda_i}(S_i) \colon i\in\mathscr{I}\right\}$ with the Tychonoff topology is an $\omega$-bounded semitopological semigroup.
\end{theorem}

Theorems~\ref{theorem-2.1} and~\ref{theorem-2.9} imply the following:

\begin{corollary}\label{corollary-2.10}
Let $\left\{ B^0_{\lambda_i}(S_i) \colon i\in\mathscr{I}\right\}$ be a non-empty family of Hausdorff totally countably compact topological Brandt $\lambda^0_i$-extension of Hausdorff $\omega$-bounded semitopological monoids. Then the direct product $\prod\left\{ B^0_{\lambda_i}(S_i) \colon i\in\mathscr{I}\right\}$ with the Tychonoff topology is an $\omega$-bounded semitopological semigroup.
\end{corollary}

Later we shall use the following theorem from \cite{GutikPavlyk2013a}:

\begin{theorem}[{\cite[Theorem~15]{GutikPavlyk2013a}}]\label{theorem-2.11}
For any semiregular pseudocompact semitopological monoid $(S,\tau)$ with zero and for any cardinal $\lambda\geqslant 1$ there exists a unique semiregular pseudocompact  topological Brandt $\lambda^0$-extension $\left(B^0_{\lambda}(S),\tau_{B}^S\right)$ of $(S,\tau)$ in the class of semitopological semigroups, and the topology $\tau_{B}^S$ is generated by the base $\mathscr{B}_B=\bigcup\left\{\mathscr{B}_B(t)\colon t\in B^0_{\lambda}(S)\right\}$, where:
\begin{itemize}
    \item[$(i)$] $\mathscr{B}_B(t)=\big\{(U(s)\setminus\{ 0_S\})_{\alpha,\beta} \colon U(s)\in \mathscr{B}_S(s)\big\}$, where $t=(\alpha,s,\beta)$ is a non-zero element of    $B^0_{\lambda}(S)$, $\alpha,\beta\in\lambda$;

    \item[$(ii)$] $\mathscr{B}_B(0)=\Big\{U_{A}(0)= \bigcup_{(\alpha,\beta)\in(\lambda\times\lambda) \setminus A}S_{\alpha,\beta}\cup \bigcup_{(\gamma,\delta)\in A} (U(0_S))_{\gamma,\delta} \colon A \hbox{~is a finite subset of~} \lambda\times\lambda$ and  $U(0_S)\in\mathscr{B}_S(0_S)\Big\}$, where $0$ is the zero of $B^0_{\lambda}(S)$,
\end{itemize}
and $\mathscr{B}_S(s)$ is a base of the topology $\tau$ at the point $s\in S$.
\end{theorem}

\begin{theorem}\label{theorem-2.12}
Let $\left\{ B^0_{\lambda_i}(S_i) \colon i\in\mathscr{I}\right\}$ be a non-empty family of semiregular pseudocompact topological Brandt $\lambda^0_i$-extension of semiregular pseudocompact semitopological monoids such that the Tychonoff product $\prod\left\{ S_i \colon i\in\mathscr{I}\right\}$ is a pseudocompact space. Then the direct product $\prod\left\{ B^0_{\lambda_i}(S_i) \colon i\in\mathscr{I}\right\}$ with the Tychonoff topology is a semiregular pseudocompact semitopological semigroup.
\end{theorem}

\begin{proof}
Since by Lemma~20 from \cite{Ravsky-arxiv1003.5343v5} the Tychonoff product of regular open sets is regular open we obtain that Tychonoff product of semiregular topological spaces is semiregular.

Let for every $i\in\mathscr{I}$, $\mathscr{A}(\lambda_i\times\lambda_i)$ be a space and $g_i\colon\mathscr{A}(\lambda_i\times\lambda_i)\times S_i\to B^0_{\lambda_i}(S_i)$ be the map defined in the proof of Theorem~\ref{theorem-2.4}. Theorem~\ref{theorem-2.11} implies that the map $g_i$ is continuous for every $i\in\mathscr{I}$. Since the space $\prod_{i\in\mathscr{I}}\mathscr{A}(\lambda_i\times\lambda_i)\times S_i$ is homeomorphic to $\prod_{i\in\mathscr{I}}\mathscr{A}(\lambda_i\times\lambda_i)\times \prod_{i\in\mathscr{I}} S_i$, Theorem~3.2.4 from \cite{Engelking1989} and Corollary~3.3~\cite{GutikRavsky2013??} imply that the Tychonoff product $\prod_{i\in\mathscr{I}}\mathscr{A}(\lambda_i\times\lambda_i)\times S_i$ is a pseudocompact space. Then by Theorem~\ref{theorem-2.11} and Proposition~2.3.6 of \cite{Engelking1989} the map $g\colon \prod_{i\in\mathscr{I}}\mathscr{A}(\lambda_i\times\lambda_i)\times S_i\to \prod_{i\in\mathscr{I}}B^0_{\lambda_i}(S_i)$ defined by the formula $g= \prod_{i\in\mathscr{I}}g_i$ is continuous, and hence the direct product $\prod\left\{B^0_{\lambda_i}(S_i)\colon i\in\mathscr{I}\right\}$ with the Tychonoff topology is a semiregular pseudocompact semitopological semigroup.
\end{proof}

\begin{proposition}\label{proposition-2.13}
Let $X,Y$ be Hausdorff countably pracompact spaces. Then the product $X\times Y$ is countably pracompact provided $Y$ is a $k$-space or sequentially compact.
\end{proposition}

\begin{proof} Let the space $X$ be countably compact at its dense subset $D_X$ and
the space $Y$ be countably compact at its dense subset $D_Y$. The set $D_X\times D_Y$ is a dense subset of the space $X\times Y$. We claim that the space $X\times Y$ is
countably compact at the set $D_X\times D_Y$. Indeed, let $A=\{(x_s,y_s)\colon s\in S\}$ be an infinite subset of the set $D_X\times D_Y$ such that $(x_s,y_s)\ne (x_{s'},y_{s'})$ provided $s\ne s'$. Assume that the set $A$ has no accumulation point in the space $X$.

If $Y$ is a $k$-space then Lemma~3.10.12 from \cite{Engelking1989} implies that
there exists an infinite subset $S_0\subset S$ such that either the set $\{x_s\colon s\in S_0\}$ or the set $\{y_s\colon s\in S_0\}$ has no accumulation point. Then this set is finite. Without loss of generality, we can assume that there are a point $x\in X$ and an infinite subset $S_1$ of the set $S_0$ such that $x_s=x$ for each index $s\in S_1$. Since the space $Y$ is countably compact at the set $D_Y$, there exists an accumulation point $y\in Y$ of the set $\{y_s\colon s\in S_1\}$. Then the point $(x,y)$ is an accumulation point of the set $\{(x_s,y_s)\colon s\in S_1\}$, a contradiction.

If $Y$ is a sequentially compact space then the proof of the claim is similar to the proof of Theorem 3.10.36 from \cite{Engelking1989}.
\end{proof}

Proposition~\ref{proposition-2.13} implies the following corollary:

\begin{corollary}\label{corollary-2.14}
The product $X\times Y$ of Hausdorff countably pracompact space $X$ and compactum $Y$ is countably pracompact.
\end{corollary}

\begin{theorem}\label{theorem-2.15}
Let $\left\{ B^0_{\lambda_i}(S_i) \colon i\in\mathscr{I}\right\}$ be a non-empty family of semiregular countably pracompact topological Brandt $\lambda^0_i$-extension of countably pracompact semiregular semitopological monoids such that the Tychonoff product $\prod\left\{ S_i \colon i\in\mathscr{I}\right\}$ is a countably pracompact space. Then the direct product \linebreak $\prod\left\{ B^0_{\lambda_i}(S_i) \colon i\in\mathscr{I}\right\}$ with the Tychonoff topology is a semiregular countably pracompact semitopological semigroup.
\end{theorem}

\begin{proof}
Let for every $i\in\mathscr{I}$, $\mathscr{A}(\lambda_i\times\lambda_i)$ be a space and $g_i\colon\mathscr{A}(\lambda_i\times\lambda_i)\times S_i\to B^0_{\lambda_i}(S_i)$ be a map defined in the proof of Theorem~\ref{theorem-2.4}. Theorem~\ref{theorem-2.11} implies that the map $g_i$ is continuous for every $i\in\mathscr{I}$. Since the space $\prod_{i\in\mathscr{I}}\mathscr{A}(\lambda_i\times\lambda_i)\times S_i$ is homeomorphic to $\prod_{i\in\mathscr{I}}\mathscr{A}(\lambda_i\times\lambda_i)\times \prod_{i\in\mathscr{I}} S_i$, Theorem~3.2.4 from \cite{Engelking1989} and Corollary~\ref{corollary-2.14} imply that the Tychonoff product $\prod_{i\in\mathscr{I}}\mathscr{A}(\lambda_i\times\lambda_i)\times S_i$ is a countably pracompact space. Then by Theorem~\ref{theorem-2.11} and Proposition~2.3.6 of \cite{Engelking1989} the map $g\colon \prod_{i\in\mathscr{I}}\mathscr{A}(\lambda_i\times\lambda_i)\times S_i\to \prod_{i\in\mathscr{I}}B^0_{\lambda_i}(S_i)$ defined by the formula $g= \prod_{i\in\mathscr{I}}g_i$ is continuous, and since by Lemma~8 from \cite{GutikPavlyk2013a} every continuous image of a countably pracompact space is countably pracompact, we see that the direct product $\prod\{B^0_{\lambda_i}(S_i)\colon i\in\mathscr{I}\}$ with the Tychonoff topology is a semiregular pseudocompact semitopological semigroup.
\end{proof}

Since for any semitopological monoid $(S,\tau)$ with zero and for any finite cardinal $\lambda\geqslant 1$  there exists a unique topological Brandt $\lambda^0$-extension
$\left(B^0_{\lambda}(S),\tau_{B}\right)$ of $(S,\tau)$ in the class of semitopological semigroups, the proof of the following theorem is similar to the proofs of Theorems~\ref{theorem-2.12} and \ref{theorem-2.15}.

\begin{theorem}\label{theorem-2.16}
Let $\left\{ B^0_{\lambda_i}(S_i) \colon i\in\mathscr{I}\right\}$ be a non-empty family of Hausdorff pseudocompact (countably pracompact) topological Brandt $\lambda^0_i$-extension of Hausdorff pseudocompact (countably pracompact) semitopological monoids such that the Tychonoff product $\prod\left\{ S_i \colon i\in\mathscr{I}\right\}$ is a Hausdorff pseudocompact (countably pracompact) space and every cardinal $\lambda_i$, $i\in\mathscr{I}$, is non-zero and finite. Then the direct product $\prod\left\{ B^0_{\lambda_i}(S_i) \colon i\in\mathscr{I}\right\}$ with the Tychonoff topology is a Hausdorff pseudocompact (countably pracompact) semitopological semigroup.
\end{theorem}

By Theorem~2.19 of \cite{GutikRavsky2013??} we have that a topological Brandt $\lambda^0$-extension $\left(B^0_{\lambda}(S), \tau_{B}\right)$ of a topological monoid $(S,\tau_S)$ with zero in the class of Hausdorff topological semigroups is pseudocompact if and only if cardinal $\lambda$ is finite and the space $(S,\tau_S)$ is pseudocompact. Hence Theorem~\ref{theorem-2.16} and Theorem~2.19 of \cite{GutikRavsky2013??} imply the following:

\begin{theorem}\label{theorem-2.17}
Let $\left\{ B^0_{\lambda_i}(S_i) \colon i\in\mathscr{I}\right\}$ be a non-empty family of Hausdorff pseudocompact (countably pracompact) topological Brandt $\lambda^0_i$-extension of Hausdorff pseudocompact (countably pracompact) topological monoids in the class of Hausdorff topological semigroups such that the Tychonoff product $\prod\left\{ S_i \colon i\in\mathscr{I}\right\}$ is a Hausdorff pseudocompact (countably pracompact) space. Then the direct product $\prod\left\{ B^0_{\lambda_i}(S_i) \colon i\in\mathscr{I}\right\}$ with the Tychonoff topology is a Hausdorff pseudocompact (countably pracompact) topological semigroup.
\end{theorem}

The following lemma describes the main property of a base of the topology at zero of a Hausdorff pseudocompact topological Brandt $\lambda^0$-extension of a Hausdorff pseudocompact semitopological monoid in the class of Hausdorff semitopological semigroups.

\begin{lemma}\label{lemma-2.18}
Let $\left(B^0_{\lambda}(S),\tau_{B}^S\right)$ be any Hausdorff pseudocompact  topological Brandt $\lambda^0$-extension of a pseudocompact semitopological monoid $(S,\tau)$ with zero in the class of semitopological semigroups. Then for every open neighbourhood $U(0)$ of zero $0$ in $\left(B^0_{\lambda}(S),\tau_{B}^S\right)$ there exist at most finitely many pairs of indices $(\alpha_1,\beta_1),\ldots, (\alpha_n,\beta_n)\in\lambda\times\lambda$ such that $S^*_{\alpha_i,\beta_i}\nsubseteq\operatorname{cl}_{B^0_{\lambda}(S)}(U(0))$ for any $i=1,\ldots, n$.
\end{lemma}

\begin{proof}
Suppose to the contrary: there exist an open neighbourhood $V(0)$ of zero $0$ in $\left(B^0_{\lambda}(S),\tau_{B}^S\right)$ and infinitely many pairs of indices $(\alpha_1,\beta_1),\ldots, (\alpha_n,\beta_n),\ldots\in\lambda\times\lambda$ such that $S^*_{\alpha_i,\beta_i}\nsubseteq\operatorname{cl}_{B^0_{\lambda}(S)}(U(0))$ for every positive integer $i$. Then by Proposition~1.1.1 of \cite{Engelking1989} for every positive integer $i$ there exists a non-empty open subset $W_{\alpha_i,\beta_i}$ in $\left(B^0_{\lambda}(S),\tau_{B}^S\right)$ such that $W_{\alpha_i,\beta_i}\subseteq S^*_{\alpha_i,\beta_i}$ and $V(0)\cap W_{\alpha_i,\beta_i}=\varnothing$. Hence by Lemma~3 of \cite{GutikPavlyk2013a} we have that $\left\{W_{\alpha_i,\beta_i}\colon i=1,2,3,\ldots\right\}$ is an infinite locally finite family in $\left(B^0_{\lambda}(S),\tau_{B}^S\right)$ which contradicts the pseudocompactness of the space $\left(B^0_{\lambda}(S),\tau_{B}^S\right)$. The obtained contradiction implies the statement of our lemma.
\end{proof}

Given a topological space $(X,\tau)$ Stone \cite{Stone1937} and Kat\v{e}tov \cite{Katetov} consider the topology $\tau_r$ on $X$ generated by the base consisting of all regular open sets of the space $(X,\tau)$. This topology is called the \emph{regularization} of the topology $\tau$. It is easy to see that if $(X,\tau)$ is a Hausdorff topological space then $(X,\tau_r)$ is a semiregular topological space.

\begin{example}\label{example-2.19}
Let $(S,\tau)$ be any semitopological monoid with zero. Then for any infinite cardinal $\lambda$ we define a topology $\tau_{B}^S$ on the
Brandt $\lambda^0$-extension $\left(B^0_{\lambda}(S),\tau_{B}^S\right)$ of $(S,\tau)$ in the following way. The topology $\tau_{B}^S$ is generated by the base
$\mathscr{B}_B=\bigcup\left\{\mathscr{B}_B(t)\colon t\in
B^0_{\lambda}(S)\right\}$, where:
\begin{itemize}
    \item[$(i)$] $\mathscr{B}_B(t)=\big\{(U(s) \setminus\{ 0_S\})_{\alpha,\beta}\colon U(s)\in \mathscr{B}_S(s)\big\}$, where $t=(\alpha,s,\beta)$ is a non-zero element of    $B^0_{\lambda}(S)$, $\alpha,\beta\in\lambda$;

    \item[$(ii)$] $\mathscr{B}_B(0)=\Big\{U_{A}(0)= \bigcup_{(\alpha,\beta)\in(\lambda\times\lambda) \setminus A}S_{\alpha,\beta}\cup \bigcup_{(\gamma,\delta)\in A} (U(0_S))_{\gamma,\delta} \colon A \hbox{~is a finite subset of~} \lambda\times\lambda$ and  $U(0_S)\in\mathscr{B}_S(0_S)\Big\}$, where $0$ is the zero of $B^0_{\lambda}(S)$,
\end{itemize}
and $\mathscr{B}_S(s)$ is a base of the topology $\tau$ at the point $s\in S$.

We observe that the space $\left(B^0_{\lambda}(S),\tau_{B}^S\right)$ is Hausdorff (resp., regular, Tychonoff, normal) if and only if the space $(S,\tau)$ is Hausdorff (resp., regular, Tychonoff, normal) (see: Propositions~21 and 22 in \cite{GutikPavlyk2013a}).
\end{example}

\begin{proposition}\label{proposition-2.20}
Let $\lambda$ be any infinite cardinal. If $(S,\tau)$ is a Hausdorff semitopological monoid with zero then $\left(B^0_{\lambda}(S),\tau_{B}^S\right)$ is a Hausdorff semitopological semigroup. Moreover, the space $(S,\tau)$ is pseudocompact if and only if so is $\left(B^0_{\lambda}(S),\tau_{B}^S\right)$.
\end{proposition}

\begin{proof}
The Hausdorffness of the space $\left(B^0_{\lambda}(S),\tau_{B}^S\right)$ follows from Proposition~21 from \cite{GutikPavlyk2013a}.

Let $a$ and $b$ are arbitrary elements of $S$ and $W(ab)$, $U(a)$, $V(b)$ be arbitrary open neighborhoods of the elements $ab$, $a$ and $b$, respectively, such that $U(a)\cdot b\subseteq W(ab)$ and $a\cdot V(b)\subseteq W(ab)$. Then we have that the following conditions hold:
\begin{itemize}
  \item[$(i)$] $(U(a))_{\alpha,\beta}\cdot (\beta,b,\gamma)\subseteq (W(ab))_{\alpha,\gamma}$;
  \item[$(ii)$] $(\alpha,a,\beta)\cdot (V(b))_{\beta,\gamma}\subseteq (W(ab))_{\alpha,\gamma}$;
  \item[$(iii)$] if $\beta\neq\gamma$ then $(U(a))_{\alpha,\beta}\cdot (\gamma,b,\delta)=\{0\}\subseteq W_{A}(0)$ and $(\alpha,a,\beta)\cdot (V(b))_{\gamma,\delta}=\{0\}\subseteq W_{A}(0)$ for every finite subset $A$ of $\lambda\times\lambda$ and every $W(0_S)\in\mathscr{B}_S(0_S)$;
  \item[$(iv)$] $W_{A}(0)\cdot 0=\{0\}\subseteq W_{A}(0)$ and $0\cdot W_{A}(0)=\{0\}\subseteq W_{A}(0)$ for every finite subset $A$ of $\lambda\times\lambda$ and every $W(0_S)\in\mathscr{B}_S(0_S)$;
  \item[$(v)$] $(U(a))_{\alpha,\beta}\cdot 0=\{0\}\subseteq W_{A}(0)$ and $0\cdot (V(b))_{\beta,\gamma}=\{0\}\subseteq W_{A}(0)$ for every finite subset $A$ of $\lambda\times\lambda$ and every $W(0_S)\in\mathscr{B}_S(0_S)$;
  \item[$(vi)$] $(\alpha,a,\beta)\cdot V_{B}(0)\subseteq W_{A}(0)$ for every finite subset $\left\{\alpha_1,\ldots,\alpha_k\right\}\subset\lambda$ and every $W(0_S)\in\mathscr{B}_S(0_S)$, where $A=\left\{\alpha,\alpha_1,\ldots,\alpha_k\right\}\times \left\{\alpha_1,\ldots,\alpha_k\right\}$ and $B=\left\{(\beta,\alpha_1),\ldots,(\beta,\alpha_k)\right\}$;
  \item[$(vii)$] $ V_{B}(0)\cdot(\alpha,a,\beta)\subseteq W_{A}(0)$ for every finite subset $\left\{\alpha_1,\ldots,\alpha_k\right\}\subset\lambda$ and every $W(0_S)\in\mathscr{B}_S(0_S)$, where $A=\left\{\alpha_1,\ldots,\alpha_k\right\}\times \left\{\beta,\alpha_1,\ldots,\alpha_k\right\}$ and $B=\left\{(\alpha_1,\alpha),\ldots,(\alpha_k,\alpha)\right\}$,
\end{itemize}
for each $\alpha,\beta,\gamma,\delta\in\lambda$. This completes the proof of separate continuity of the semigroup operation in $\left(B^0_{\lambda}(S),\tau_{B}^S\right)$.

The implication $(\Leftarrow)$ of the last statement follows from Lemma~9 of \cite{GutikPavlyk2013a}.  To show the converse implication assume that $\{U_i\colon i\in\mathscr{I}\}$ is any locally finite family of open subsets of $\left(B^0_{\lambda}(S),\tau_{B}^S\right)$. Without loss of generality we can assume that $0\notin U_i$ for any $i\in\mathscr{I}$. Then the definition of the base of the topology $\tau_{B}^S$ at zero implies that there exists a finite family of pairs of indices $\left\{(\alpha_1,\beta_1),\ldots,(\alpha_k,\beta_k)\right\}\subset \lambda\times\lambda$ such that almost all elements of the family $\{U_i\colon i\in\mathscr{I}\}$ are contained in the set $S^*_{\alpha_1,\beta_1}\cup\cdots\cup S^*_{\alpha_k,\beta_k}$. Since a union of a finite family of pseudocompact spaces is pseudocompact, $S_{\alpha_1,\beta_1}\cup\cdots\cup S_{\alpha_k,\beta_k}$ with the topology induced from $\left(B^0_{\lambda}(S),\tau_{B}^S\right)$ is pseudocompact space. This implies that the family $\{U_i\colon i\in\mathscr{I}\}$ is finite.
\end{proof}

\begin{example}\label{example-2.21}
Let $\lambda$ be any infinite cardinal. Let $(S,\tau_S)$ be a Hausdorff pseudocompact semitopological monoid with zero $0_S$ and $\left(B^0_{\lambda}(S),\tau^0_{B_S}\right)$ be a pseudocompact topological Brandt $\lambda^0$-extension of $(S,\tau_S)$ in the class of Hausdorff semitopological semigroups.

For every open neighbourhood $U(0)$ of zero in $\left(B^0_{\lambda}(S),\tau^0_{B_S}\right)$ we put
\begin{equation*}
    A_{U(0)}=\left\{(\alpha,\beta)\in\lambda\times\lambda\colon S_{\alpha,\beta}\nsubseteq\operatorname{cl}_{B^0_{\lambda}(S)}(U(0)) \right\}.
\end{equation*}

Let $\pi_{B_S}\colon B_{\lambda}(S)\rightarrow B^0_{\lambda}(S)= B_{\lambda}(S)/\mathscr{J}$ be the natural homomorphisms, where
\begin{equation*}
 \mathscr{J}=\{ 0\}\cup\{(\alpha, 0_S, \beta)\colon 0_S \hbox{~is zero of~} S\}
\end{equation*}
is an ideal of the semigroup $B_\lambda(S)$.

We generate a topology $\widehat{\tau}_{B_S}$ on the Brandt
$\lambda$-extension $B_{\lambda}(S)$ by a base
$\widehat{\mathscr{B}}_B=\bigcup\left\{\widehat{\mathscr{B}}_B(t)\colon t\in
B_{\lambda}(S)\right\}$, where:
\begin{itemize}
    \item[$(i)$] $\widehat{\mathscr{B}}_B(\alpha,s,\beta)=\big\{(U(s))_{\alpha,\beta}\colon U(s)\in \mathscr{B}_S(s)\big\}$, for all $s\in S$ and $\alpha,\beta\in\lambda$;

    \item[$(ii)$] $\widehat{\mathscr{B}}_B(0)=\Big\{U_\pi(0)=\pi^{-1}(U(0))\setminus \bigcup_{(\alpha,\beta)\in A_{U(0)}} S_{\alpha,\beta} \colon U(0) \hbox{~is an element of a base of the topology~} \tau^0_{B_S}$ at zero $0$ of $B^0_{\lambda}(S)\Big\}$,
\end{itemize}
and $\mathscr{B}_S(s)$ is a base of the topology $\tau$ at the point $s\in S$.
\end{example}

\begin{proposition}\label{proposition-2.22}
Let $\lambda$ be any infinite cardinal. Let  $\left(B^0_{\lambda}(S),\tau^0_{B_S}\right)$ be a pseudocompact topological Brandt $\lambda^0$-extension of a Hausdorff pseudocompact semitopological monoid with zero $(S,\tau_S)$ in the class of Hausdorff semitopological semigroup. Then $\left(B_{\lambda}(S),\widehat{\tau}_{B_S}\right)$ is a Hausdorff semitopological semigroup. Moreover, the space $(S,\tau)$ is pseudocompact if and only if so is $\left(B_{\lambda}(S),\widehat{\tau}_{B_S}\right)$.
\end{proposition}

\begin{proof}
We observe that simple verifications show that the natural homomorphism $\pi_{B_S}\colon \left(B_{\lambda}(S),\widehat{\tau}_{B_S}\right)\rightarrow \left(B^0_{\lambda}(S),\tau^0_{B_S}\right)$ is a continuous map.

Let $a$ and $b$ be arbitrary elements of the semitopological semigroup $(S,\tau_S)$ and $W(ab)$, $U(a)$, $V(b)$ be arbitrary open neighborhoods of the elements $ab$, $a$ and $b$, respectively, such that $U(a)\cdot b\subseteq W(ab)$ and $a\cdot V(b)\subseteq W(ab)$. Then the following conditions hold:
\begin{itemize}
  \item[$(i)$] $(U(a))_{\alpha,\beta}\cdot (\beta,b,\gamma)\subseteq (W(ab))_{\alpha,\gamma}$;
  \item[$(ii)$] $(\alpha,a,\beta)\cdot (V(b))_{\beta,\gamma}\subseteq (W(ab))_{\alpha,\gamma}$;
  \item[$(iii)$] if $\beta\neq\gamma$ then $(U(a))_{\alpha,\beta}\cdot (\gamma,b,\delta)=\{0\}\subseteq U_\pi(0)$ and $(\alpha,a,\beta)\cdot (V(b))_{\gamma,\delta}=\{0\}\subseteq U_\pi(0)$ for every open neighbourhood $U(0)$ of zero in $\left(B^0_{\lambda}(S),\tau^0_{B_S}\right)$;
  \item[$(iv)$] $U_\pi(0)\cdot 0=\{0\}\subseteq U_\pi(0)$ and $0\cdot U_\pi(0)=\{0\}\subseteq U_\pi(0)$ for every open neighbourhood $U(0)$ of zero in $\left(B^0_{\lambda}(S),\tau^0_{B_S}\right)$;
  \item[$(v)$] $(U(a))_{\alpha,\beta}\cdot 0=\{0\}\subseteq U_\pi(0)$ and $0\cdot (V(b))_{\beta,\gamma}=\{0\}\subseteq U_\pi(0)$ for every open neighbourhood $U(0)$ of zero in $\left(B^0_{\lambda}(S),\tau^0_{B_S}\right)$;

  \item[$(vi)$] $(\alpha,a,\beta)\cdot U_\pi(0)\subseteq W_\pi(0)$ in $\left(B_{\lambda}(S),\widehat{\tau}_{B_S}\right)$ for $U_\pi(0), W_\pi(0)\in \widehat{\mathscr{B}}_B(0)$< where $U(0)$ and $W(0)$ are elements of a base of the topology $\tau^0_{B_S}$ at zero $0$ of $B^0_{\lambda}(S)$ such that $(\alpha,a,\beta)\cdot U(0)\subseteq W(0)$;
  \item[$(vii)$] $U_\pi(0)\cdot (\alpha,a,\beta)\subseteq W_\pi(0)$ in $\left(B_{\lambda}(S),\widehat{\tau}_{B_S}\right)$ for $U_\pi(0), W_\pi(0)\in \widehat{\mathscr{B}}_B(0)$< where $U(0)$ and $W(0)$ are elements of a base of the topology $\tau^0_{B_S}$ at zero $0$ of $B^0_{\lambda}(S)$ such that $U(0)\cdot(\alpha,a,\beta)\subseteq W(0)$,
\end{itemize}
for each $\alpha,\beta,\gamma,\delta\in\lambda$.

The proof of the last statement is similar to the proof of the second statement of Proposition~\ref{proposition-2.20}.
\end{proof}

\begin{remark}\label{remark-2.23}
Also, we may consider the semitopological semigroup $\left(B_{\lambda}(S),\widehat{\tau}_{B_S}\right)$ as a topological Brandt $\lambda^0$-extension of a Hausdorff pseudocompact semitopological monoid $T=S\sqcup 0_T$ with ``new'' isolated zero $0_T$.
\end{remark}

\begin{theorem}\label{theorem-2.24}
Let $\big\{ \big(B^0_{\lambda_i}(S_i),\tau^0_{B(S_i)}\big) \colon i\in\mathscr{I}\big\}$ be a non-empty family of Hausdorff pseudocompact to\-pological Brandt $\lambda_i^0$-extension of Hausdorff pseudocompact semitopological monoids with zero such that the Tychonoff product $\prod\left\{ S_i \colon i\in\mathscr{I}\right\}$ is a pseudocompact space. Then the direct product\linebreak $\prod\big\{ \big(B^0_{\lambda_i}(S_i),\tau^0_{B(S_i)}\big) \colon i\in\mathscr{I}\big\}$ with the Tychonoff topology is a Hausdorff pseudocompact semitopological semigroup.
\end{theorem}

\begin{proof}
We consider two cases: 1) $\lambda_i$ is finite cardinal, and 2) $\lambda_i$ is infinite cardinal, $i\in\mathscr{I}$.

1) Let $i\in\mathscr{I}$ be an index such that $\lambda_i$ is infinite cardinal. Then we put $\widehat{\tau}_{B(S_i)}$ is the topology on the Brandt $\lambda_i$-extension $B_{\lambda_i}(S_i)$ defined in Example~\ref{example-2.21}. Then by Proposition~\ref{proposition-2.22}, $\left(B_{\lambda_i}(S_i),\widehat{\tau}_{B(S_i)}\right)$ is a Hausdorff pseudocompact semitopological semigroup. By Remark~\ref{remark-2.23} we have that the semitopological semigroup $\left(B_{\lambda_i}(S_i),\widehat{\tau}_{B(S_i)}\right)$ is a topological Brandt $\lambda^0_i$-extension of a Hausdorff pseudocompact semitopological monoid $T_i=S\sqcup 0_{T_i}$ with isolated zero $0_{T_i}$. By $\tau_i$ we denote the topology of the space $T_i$. Let $\tau_{B}^{T_i}$ be the topology on the Brandt $\lambda^0$-extension $\left(B^0_{\lambda_i}(T_i),\tau_{B}^{T_i}\right)$ of $(T_i,\tau_i)$ defined in Example~\ref{example-2.19}. Next we algebraically identify the semigroup $B^0_{\lambda_i}(T_i)$ with the Brandt $\lambda_i$-extension $B_{\lambda_i}(S_i)$ and the topology $\tau_{B}^{T_i}$ on $B_{\lambda_i}(S_i)$ we shall denote by $\tau_{B}^{S_i}$.

2) Let $i\in\mathscr{I}$ be an index such that $\lambda_i$ is finite cardinal. We put $T_i=S\sqcup 0_{T_i}$ with isolated zero $0_{T_i}$. It is obvious that the semitopological semigroup $T_i$ is pseudocompact if and only if so is the space $S_i$. Then by Theorem~7 from \cite{GutikPavlyk2013a} there exists the unique topological Brandt $\lambda_i^0$-extension $\left(B^0_{\lambda_i}(T_i),\widehat{\tau}_{B(T_i)}\right)$ of the semitopological monoid $T_i$ in the class of semitopological semigroups. Also, Theorem~7 from \cite{GutikPavlyk2013a} implies that the topological space $\left(B^0_{\lambda_i}(T_i),\widehat{\tau}_{B(T_i)}\right)$ is homeomorphic to the topological sum of topological copies of the space $S_i$ and isolated zero, and hence we obtain that the space $\left(B^0_{\lambda_i}(T_i),\widehat{\tau}_{B(T_i)}\right)$ is pseudocompact if and only if so is the space $S_i$. Next we algebraically identify the semigroup $B^0_{\lambda_i}(T_i)$ with the Brandt $\lambda_i$-extension $B_{\lambda_i}(S_i)$ and the topology $\widehat{\tau}_{B(T_i)}$ on $B_{\lambda_i}(S_i)$ we shall denote by $\widehat{\tau}_{B(S_i)}$. Also in this case (when $\lambda_i$ is a finite cardinal) we put $\tau_{B}^{S_i}=\widehat{\tau}_{B(S_i)}$.

Then the definitions of topologies $\widehat{\tau}_{B(S_i)}$ and $\tau_{B}^{S_i}$ on $B_{\lambda_i}(S_i)$ imply that for every index $i\in\mathscr{I}$ the identity map $\widehat{\operatorname{id}}_i\colon \left(B_{\lambda_i}(S_i),\widehat{\tau}_{B(S_i)}\right)\rightarrow \left(B_{\lambda_i}(S_i),\tau_{B}^{S_i}\right)$ is continuous. Let $\tau^R_{B(S_i)}$ be the regularization of the topology $\widehat{\tau}_{B(S_i)}$ on $B_{\lambda_i}(S_i)$. Then the definition of the topology $\tau_{B}^{S_i}$ on $B_{\lambda_i}(S_i)$ implies that the identity map $\operatorname{id}^R_i\colon \left(B_{\lambda_i}(S_i),\tau_{B}^{S_i}\right)\rightarrow \big(B_{\lambda_i}(S_i),\tau^R_{B(S_i)}\big)$ is continuous. Since the pseudocompactness is preserved by continuous maps we obtain that $\big(B_{\lambda_i}(S_i),\tau^R_{B(S_i)}\big)$ is a semiregular pseudocompact space (which is not necessarily a semitopological semigroup). Also, repeating the proof of Theorem~\ref{theorem-2.12} for our case, we get that the Tychonoff product $\prod_{i\in\mathscr{I}} \left(B_{\lambda_i}(S_i),\tau_{B}^{S_i}\right)$ is a pseudocompact space. Then the Cartesian products $\prod_{i\in\mathscr{I}}\widehat{\operatorname{id}}_i\colon \prod_{i\in\mathscr{I}} \left(B_{\lambda_i}(S_i),\widehat{\tau}_{B(S_i)}\right) \rightarrow \prod_{i\in\mathscr{I}}\left(B_{\lambda_i}(S_i),\tau_{B}^{S_i}\right)$ and $\prod_{i\in\mathscr{I}}\operatorname{id}^R_i\colon \prod_{i\in\mathscr{I}} \left(B_{\lambda_i}(S_i),\tau_{B}^{S_i}\right)\rightarrow \prod_{i\in\mathscr{I}}\big(B_{\lambda_i}(S_i),\tau^R_{B(S_i)}\big)$ are continuous maps. This implies that $\prod_{i\in\mathscr{I}}\big(B_{\lambda_i}(S_i),\tau^R_{B(S_i)}\big)$ is a pseudocompact space. Then by Lemma~20 of \cite{Ravsky-arxiv1003.5343v5} the regularization of the product $\prod_{i\in\mathscr{I}} \left(B_{\lambda_i}(S_i),\widehat{\tau}_{B(S_i)}\right)$ coincides with $\prod_{i\in\mathscr{I}}\big(B_{\lambda_i}(S_i),\tau^R_{B(S_i)}\big)$ and hence by  Lemma~3 of \cite{Ravsky-arxiv1003.5343v5} we have that the space $\prod_{i\in\mathscr{I}} \left(B_{\lambda_i}(S_i),\widehat{\tau}_{B(S_i)}\right)$ is pseudocompact.

Let $\pi_{B_S^i}\colon B_{\lambda_i}(S_i)\rightarrow B^0_{\lambda_i}(S_i)= B_{\lambda_i}(S_i)/\mathscr{J_i}$ be the natural homomorphism, where $\mathscr{J_i}=\{ 0_i\}\cup\{(\alpha, 0_{S_i}, \beta)\colon 0_{S_i} \hbox{~is zero of~} S_i\}$ is an ideal of the semigroup $B_{\lambda_i}(S_i)$. Then the natural homomorphism \break $\pi_{B_S^i}\colon  \left(B_{\lambda_i}(S_i),\widehat{\tau}_{B(S_i)}\right) \rightarrow \big(B^0_{\lambda_i}(S_i),\tau^0_{B(S_i)}\big)$ is a continuous map. This implies that the product \break $\prod_{i\in\mathscr{I}}\pi_{B_S^i} \colon  \prod_{i\in\mathscr{I}} \left(B_{\lambda_i}(S_i),\widehat{\tau}_{B(S_i)}\right) \rightarrow \prod_{i\in\mathscr{I}} \big(B^0_{\lambda_i}(S_i),\tau^0_{B(S_i)}\big)$ is a continuous map, and hence we get that the Tychonoff product $\prod_{i\in\mathscr{I}} \big(B^0_{\lambda_i}(S_i),\tau^0_{B(S_i)}\big)$ is a pseudocompact space.
\end{proof}

\begin{proposition}\label{proposition-2.25}
Each $H$-closed space is pseudocompact.
\end{proposition}

\begin{proof}
Let $X$ be an $H$-closed space. Assume that the space $X$ is not pseudocompact. Then there exists an infinite locally finite family $\mathscr{U}$ of non-empty open subsets of the space $X$. Since the family $\mathscr{U}$ is locally finite, each point $x\in X$ has an open neighbourhood $U_x$ intersecting only finitely many members of the family $\mathscr{U}$. Since the space $X$ is $H$-closed and $\{U_x\colon x\in X\}$ is an open cover of the space $X$, by Exercise~3.12.5(4) from \cite{Engelking1989} (also see \cite[Chapter~III, Theorem~4]{AlexandroffUrysohn1929}) there exists a finite subset $F$ of the space $X$ such that $X=\bigcup\{\operatorname{cl}_X(U_x)\colon x\in F\}$. But then the set $X$, as the union of the finite family $\{\operatorname{cl}_X(U_x)\colon x\in F\}$ intersects only finitely many members of the family $\mathscr{U}$, a contradiction.
\end{proof}

Let $\lambda$ be any cardinal $\geqslant 1$ and $S$ be any semigroup. We shall say that a subset $\Phi\subset B_{\lambda}^0(S)$ has the \emph{$\lambda$-finite property} in $B_{\lambda}^0(S)$, if $\Phi\cap S^*_{\alpha,\beta}$ is finite for all $\alpha,\beta\in\lambda$ and $\Phi\not\ni 0$, where $0$ is zero of $B_{\lambda}^0(S)$.

\begin{example}\label{example-2.26}
Let $\lambda$ be an infinite cardinal and $\mathbb{T}$ be the unit circle with the usual multiplication of complex numbers and the usual topology $\tau_{\mathbb{T}}$. It is obvious that $(\mathbb{T},\tau_{\mathbb{T}})$ is a topological group. The base of the topology $\tau_B^{\textsf{fin}}$ on the Brandt semigroup $B_\lambda(\mathbb{T})$ we define as follows:
\begin{enumerate}
  \item[1)] for every non-zero element $(\alpha,x,\beta)$ of the semigroup $B_\lambda(\mathbb{T})$ the family
      \begin{equation*}
        \mathscr{B}_{(\alpha,x,\beta)}=\left\{(\alpha,U(x),\beta)\colon U(x)\in \mathscr{B}_\mathbb{T}(x)\right\},
      \end{equation*}
      where $\mathscr{B}_\mathbb{T}(x)$ is a base of the topology $\tau_\mathbb{T}$ at the point $x\in \mathbb{T}$, is the base of the topology $\tau_B^{\textsf{fin}}$ at $(\alpha,x,\beta)\in B_\lambda(\mathbb{T})$;
  \item[2)] the family
      \begin{equation*}
        \mathscr{B}_{0}=\left\{U(\alpha_1,\beta_1;\ldots;\alpha_n,\beta_n; F)\colon \alpha_1,\beta_1,\ldots,\alpha_n,\beta_n\in\lambda, n\in\mathbb{N}, F \hbox{~has the~} \lambda\hbox{-finite property in~} B_{\lambda}^0(S)\right\},
      \end{equation*}
      where
      \begin{equation*}
        U(\alpha_1,\beta_1;\ldots;\alpha_n,\beta_n; F)= B_\lambda(\mathbb{T})\setminus\left(\mathbb{T}_{\alpha_1,\beta_1} \cup\cdots\cup\mathbb{T}_{\alpha_n,\beta_n}\cup F\right),
      \end{equation*}
      is the base of the topology $\tau_B^{\textsf{fin}}$ at zero $0\in B_\lambda(\mathbb{T})$.
\end{enumerate}

Simple verifications show that $(B_\lambda(\mathbb{T}),\tau_B^{\textsf{fin}})$ is a non-semiregular Hausdorff pseudocompact topological space for every infinite cardinal $\lambda$. Next we shall show that the semigroup operation on $(B_\lambda(\mathbb{T}),\tau_B^{\textsf{fin}})$ is separately continuous. The proof of the separate continuity of the semigroup operation in the cases $0\cdot 0$ and $(\alpha,x,\beta)\cdot(\gamma,y,\delta)$, where $\alpha,\beta,\gamma,\delta\in\lambda$ and $x,y\in \mathbb{T}$, is trivial, and hence we only consider the following cases:
\begin{equation*}
     (\alpha,x,\beta)\cdot 0 \qquad \hbox{and} \qquad 0\cdot(\alpha,x,\beta).
\end{equation*}

For arbitrary $\alpha,\beta\in\lambda$ and $\Phi\subset B_\lambda(\mathbb{T})$ we denote $\Phi^{\alpha,\beta}=\Phi\cap\mathbb{T}^{\alpha,\beta}$ and put $\Phi_{\mathbb{T}}(\alpha,\beta)$ is a subset of $\mathbb{T}$ such that $(\Phi_{\mathbb{T}}(\alpha,\beta))_{\alpha,\beta}=\Phi\cap \mathbb{T}_{\alpha,\beta}$.

Fix an arbitrary non-zero element $(\alpha,x,\beta)\in B_\lambda(\mathbb{T})$.
Let $\Phi\subset B_{\lambda}^0(S)$ be an arbitrary subset with the $\lambda$-finite property in $B_{\lambda}^0(S)$. Since $\mathbb{T}$ is a group, there exist subsets $\Upsilon,\Psi\subset B_{\lambda}^0(S)$ with the $\lambda$-finite property in $B_{\lambda}^0(S)$ such that
\begin{equation*}
    \left(x\cdot \Upsilon_{\mathbb{T}}(\beta,\gamma)\right)_{\alpha,\gamma}= \Phi\cap \mathbb{T}_{\alpha,\gamma} \qquad \hbox{and} \qquad \left(\Psi_{\mathbb{T}}(\gamma,\alpha)\cdot x\right)_{\gamma,\beta}= \Phi\cap \mathbb{T}_{\gamma,\beta}
\end{equation*}

Then we have that
\begin{equation*}
\begin{split}
  (\alpha,x,\beta)\cdot &  U(\beta,\beta_1;\ldots;\beta,\beta_n;\alpha_1,\beta_1;\ldots;\alpha_n,\beta_n; \Upsilon)\subseteq \\
    \subseteq &\;\{0\}\cup \bigcup\left\{\mathbb{T}_{\alpha,\gamma}\setminus \left((\alpha,x,\beta)\cdot\Upsilon^{\beta,\gamma}\right)\colon \gamma\in\lambda\setminus\{\beta_1,\ldots,\beta_n\}\right\}\subseteq \\
    \subseteq &\;\{0\}\cup \bigcup\left\{\mathbb{T}_{\alpha,\gamma}\setminus \left(x\cdot\Upsilon_{\mathbb{T}}(\beta,\gamma)\right)_{\alpha,\gamma}\colon \gamma\in\lambda\setminus\{\beta_1,\ldots,\beta_n\}\right\}\subseteq \\
    \subseteq &\; U(\alpha_1,\beta_1;\ldots;\alpha_n,\beta_n; \Phi)
\end{split}
\end{equation*}
and similarly
\begin{equation*}
\begin{split}
  U&(\alpha_1,\alpha;\ldots;\alpha_n,\alpha;\alpha_1,\beta_1;\ldots;\alpha_n,\beta_n; \Psi)\cdot (\alpha,x,\beta)\subseteq \\
    &\; \subseteq \{0\}\cup \bigcup\left\{\mathbb{T}_{\gamma,\beta}\setminus \left(\Psi^{\gamma,\alpha}\cdot(\alpha,x,\beta)\right)\colon \gamma\in\lambda\setminus\{\alpha_1,\ldots,\alpha_n\}\right\}\subseteq \\
    &\; \subseteq \{0\}\cup \bigcup\left\{\mathbb{T}_{\gamma,\beta}\setminus \left(\Psi_{\mathbb{T}}(\gamma,\alpha)\cdot x\right)_{\gamma,\beta}\colon \gamma\in\lambda\setminus\{\alpha_1,\ldots,\alpha_n\}\right\}\subseteq \\
    &\; \subseteq U(\alpha_1,\beta_1;\ldots;\alpha_n,\beta_n; \Phi),
\end{split}
\end{equation*}
for every $U(\alpha_1,\beta_1;\ldots;\alpha_n,\beta_n; \Phi)\in\mathscr{B}_{0}$. This completes the proof of separate continuity of the semigroup operation in $(B_\lambda(\mathbb{T}),\tau_B^{\textsf{fin}})$.

Next we shall show that the space $(B_\lambda(\mathbb{T}),\tau_B^{\textsf{fin}})$ is not countably pracompact. Suppose to the contrary: there exists a dense subset $A$ in $(B_\lambda(\mathbb{T}),\tau_B^{\textsf{fin}})$ such that $(B_\lambda(\mathbb{T}),\tau_B^{\textsf{fin}})$ is countably compact at $A$. Then the definition of the topology $\tau_B^{\textsf{fin}}$ implies that $A\cap \mathbb{T}_{\alpha,\beta}$ is a dense subset in $\mathbb{T}_{\alpha,\beta}$ for all $\alpha,\beta\in\lambda$. We construct a subset $\Phi\subset B_\lambda(\mathbb{T})$ in the following way. For all $\alpha,\beta\in\lambda$ we fix an arbitrary point $(\alpha,x^A_{\alpha,\beta},\beta)\in A\cap \mathbb{T}_{\alpha,\beta}$ and put $\Phi=\left\{(\alpha,x^A_{\alpha,\beta},\beta)\colon \alpha,\beta\in\lambda\right\}$. Then $\Phi$ is the subset with the $\lambda$-finite property in $B_{\lambda}^0(S)$, and the definition of the topology $\tau_B^{\textsf{fin}}$ on $B_\lambda(\mathbb{T})$ implies that $\Phi$ has no an accumulation point $x$ in $(B_\lambda(\mathbb{T}),\tau_B^{\textsf{fin}})$, a contradiction.
\end{example}

Example~\ref{example-2.26} shows that there exists a Hausdorff non-semiregular pseudocompact topological Brandt $\lambda^0$-extension of a Hausdorff compact topological group with adjoined isolated zero which is not a countably pracompact space. Also, Example~18 from \cite{GutikPavlyk2013a} shows that there exists a Hausdorff non-semiregular pseudocompact topological Brandt $\lambda^0$-extension of a countable Hausdorff compact topological monoid with adjoined isolated zero which is not a countably compact space. But, as a counterpart for the $H$-closed case or the sequentially pseudocompact case we have the following

\begin{proposition}\label{proposition-2.27}
Let $S$ be semitopological monoid with zero which is an $H$-closed (resp., a sequentially pseudocompact) space. Then every Hausdorff pseudocompact topological Brandt $\lambda^0$-extension $B^0_{\lambda}(S)$ of $S$ in the class of Hausdorff semitopological semigroup is an $H$-closed (resp., a sequentially pseudocompact) space.
\end{proposition}

\begin{proof}
First we consider the case when $S$ an $H$-closed space. Suppose to the contrary that there exists a Hausdorff pseudocompact topological Brandt $\lambda^0$-extension $(B^0_{\lambda}(S),\tau_B)$ of $S$ in the class of Hausdorff semitopological semigroup such that $(B^0_{\lambda}(S),\tau_B)$ is not an $H$-closed space. Then there exists a Hausdorff topological space $X$ which contains the topological $(B^0_{\lambda}(S),\tau_B)$ as a non-closed subspace. Without loss of generality we may assume that $(B^0_{\lambda}(S),\tau_B)$ is a dense subspace of $X$ such that $X\setminus B^0_{\lambda}(S)\neq\varnothing$. Fix an arbitrary point $x\in X\setminus B^0_{\lambda}(S)$. Then we have that $U(x)\cap B^0_{\lambda}(S)\neq\varnothing$ for any open neighbourhood $U(x)$ of the point $x$ in $X$. Now, the Hausdorffness of $X$ implies that there exist open neighbourhoods $V(x)$ and $V(0)$ of $x$ and zero $0$ of the semigroup $B^0_{\lambda}(S)$ such that $V(x)\cap V(0)=\varnothing$. Then by Lemma~\ref{lemma-2.18} we obtain that there exist at most finitely many pairs of indices $(\alpha_1,\beta_1),\ldots, (\alpha_n,\beta_n)\in\lambda\times\lambda$ such that $S^*_{\alpha_i,\beta_i}\nsubseteq\operatorname{cl}_{B^0_{\lambda}(S)}(V(0))$ for any $i=1,\ldots, n$. Hence by Corollary~1.1.2 of \cite{Engelking1989}, the neighbourhood $V(x)$ intersects at most finitely many subsets $S_{\alpha,\beta}$, $\alpha,\beta\in\lambda$. Then by Lemma~2 of \cite{GutikPavlyk2013a} we get that $S_{\alpha,\beta}$ is a closed subset of $X$ for all $\alpha,\beta\in\lambda$, and hence $B^0_{\lambda}(S)$ is a closed subspace of $X$, a contradiction.

Next we suppose that $S$ is a sequentially pseudocompact space. Let $\{U_n\colon n\in\mathbb{N}\}$ be any sequence of non-empty open subsets of the space $B^0_{\lambda}(S)$. If there exists finitely many pairs of indices $(\alpha_1,\beta_1),\ldots, (\alpha_n,\beta_n)\in\lambda\times\lambda$ such that $\bigcup\{U_n\colon n\in\mathbb{N}\}\subseteq S_{\alpha_1,\beta_1}\cup\cdots\cup S_{\alpha_n,\beta_n}$ the sequential pseudocompactness of $S$ and Lemma~2 from \cite{GutikPavlyk2013a} imply that there exist a point $x\in S_{\alpha_1,\beta_1}\cup\cdots\cup S_{\alpha_n,\beta_n}$ and an infinite set $S\subset\mathbb{N}$ such that for each neighborhood $U(x)$ of the point $x$ the set $\{n\in S\colon U_n\cap U(x)=\varnothing\}$ is finite. In the other case by Lemma~\ref{lemma-2.18} we get that there exists an infinite set $S\subset\mathbb{N}$ such that for each neighborhood $U(0)$ of zero $0$ of the semigroup $B^0_{\lambda}(S)$ the set $\{n\in S\colon U_n\cap U(0)=\varnothing\}$ is finite. This completes the proof of our lemma.
\end{proof}

Since by Theorem~3 from \cite{ChevalleyFrink1941} (see also Problem~3.12.5(d) in \cite{Engelking1989}) the Tychonoff product of the non-empty family non-empty $H$-topological spaces is $H$-closed, and by Proposition~2.2 from \cite{GutikRavsky2014??}, the Tychonoff product of a non-empty family of non-empty sequentially pseudocompact spaces is sequentially pseudocompact Proposition~\ref{proposition-2.27} implies the following

\begin{corollary}\label{corollary-2.28}
Let $\big\{ \big(B^0_{\lambda_i}(S_i),\tau^0_{B(S_i)}\big) \colon i\in\mathscr{I}\big\}$ be a non-empty family of Hausdorff pseudocompact topological Brandt $\lambda_i^0$-extension of Hausdorff $H$-closed (resp., a sequentially pseudocompact) semitopological monoids with zero. Then the direct product $\prod\big\{ \big(B^0_{\lambda_i}(S_i),\tau^0_{B(S_i)}\big) \colon i\in\mathscr{I}\big\}$ with the Tychonoff topology is a Hausdorff $H$-closed (resp., a sequentially pseudocompact) semitopological semigroup.
\end{corollary}

%\textcolor[rgb]{1.00,0.00,0.00}{???????????????????}
%%%%%%%%%%%%%%%%%%%%%%%%%%%%%%%%%%%%%%%%%%%%%%%%%%%%%%%%%%%%

\end{document}